\documentclass[a4paper]{article}
\usepackage[english]{babel}
\usepackage[utf8]{inputenc}
\usepackage{paralist,url,verbatim,anysize}
\usepackage{amscd}
\usepackage{color}
\usepackage[usenames,dvipsnames,svgnames,table]{xcolor}
\definecolor{purple}{HTML}{961C8C}
\usepackage{manfnt}
\usepackage{nicefrac}
\usepackage{setspace}
\usepackage[e]{esvect}
\usepackage{amsmath,amsthm,amssymb,amsfonts}
\usepackage{graphicx, graphics}
\usepackage{hyperref}

\hypersetup{colorlinks=true, citecolor=blue}
\usepackage{epstopdf}
\usepackage[mathlines, pagewise]{lineno}
\usepackage{bbm}
\usepackage[font=small,labelfont=bf]{caption}
\usepackage[labelformat=simple]{subcaption}

\theoremstyle{plain}
\newtheorem{theorem}{\bf Theorem}[subsection]

\newtheorem*{conjecture*}{Conjecture}
\newtheorem{cor}[theorem]{Corollary}
\newtheorem{lemma}[theorem]{Lemma}
\newtheorem{prp}[theorem]{Proposition}

\newtheorem{thmmain}{\bf Theorem}

\theoremstyle{definition}

\newtheorem*{rem*}{Remark}
\newtheorem{definition}[theorem]{Definition}

\newtheorem{example}[theorem]{Example}

\newtheorem{examples}[theorem]{Examples}

\DeclareMathAlphabet{\mathpzc}{OT1}{pzc}{m}{it}
\newenvironment{frcseries}{\fontfamily{frc}\selectfont}{}

\newcommand{\conv}{\operatorname{conv} }
\newcommand{\aff}{\operatorname{aff} }

\newcommand{\F}{\operatorname{F} }

\newcommand{\R}{\mathbb{R}}

\newcommand{\e}{\varepsilon}

\newcommand{\Z}{\mathbb{Z}}

\renewcommand{\e}{\varepsilon}

\newcommand{\cm}[1]{}

\newcommand{\Q}{\mathrm{Q}}
\newcommand{\PROJ}[1]{\mathrm{PROJ}\left[{#1}\right]}
\newcommand{\dH}{\mathrm{d}_{\mathrm{H}}}

\newcommand{\FUNC}[1]{\mathrm{FUNC}\left[{#1}\right]}
\newcommand{\COOR}[1]{\mathrm{COOR}\left[{#1}\right]}
\newcommand{\kF}[1]{[{#1}\text{-facet}]}
\newcommand{\veczero}{{0}}
\newcommand{\vect}[1]{{#1}}
\newcommand{\ve}{{e}}
\newcommand\Defn[1]{\emph{\color{RubineRed}#1}}

\begin{document}

\author{
Karim A. Adiprasito \thanks{Supported by DFG within the research training group ``Methods for Discrete Structures'' (GRK1408) and by the Romanian NASR, CNCS – UEFISCDI, project PN-II-ID-PCE-2011-3-0533.}\\
\small Institut f\" ur Mathematik, FU Berlin\\
\small Arnimallee 2\\
\small 14195 Berlin, Germany\\
\small \url{adiprasito@math.fu-berlin.de}
\and
\setcounter{footnote}{2}
Arnau Padrol \thanks{Supported by the DFG Collaborative Research Center SFB/TR~109 ``Discretization in Geometry and Dynamics''} \\
\small Institut f\" ur Mathematik, FU Berlin\\
\small Arnimallee 2\\
\small 14195 Berlin, Germany\\
\small \url{arnau.padrol@fu-berlin.de}
}

\date{June 12, 2013}
\title{A universality theorem for projectively unique polytopes \\ and a conjecture of Shephard}
\maketitle

\begin{abstract}
We prove that every polytope described by algebraic coordinates is the face of a projectively unique polytope. This provides a universality property for projectively unique polytopes. Using a closely related result of Below, we construct a combinatorial type of $5$-dimensional polytope that is not realizable as a subpolytope of any stacked polytope. This disproves a classical conjecture in polytope theory, first formulated by Shephard in the seventies.
\end{abstract}

By employing a technique developed by Adiprasito and Ziegler~\cite{AZ12}, we prove the following universality theorem for projectively unique polytopes. 
\begin{thmmain}\label{thm:ratprj}
For any algebraic polytope $P$, there exists a polytope $\widehat{P}$ that is projectively unique and that contains a face projectively equivalent to $P$.
\end{thmmain}

Here, a polytope is \Defn{algebraic} if the coordinates of all of its vertices are real algebraic numbers, and a polytope $P$ in $\R^d$ is \Defn{projectively unique} if any polytope ${P}'$ in $\R^d$ combinatorially equivalent to $P$ is \Defn{projectively equivalent} to~$P$. In other words, $P$ is projectively unique if for every polytope ${P}'$ combinatorially equivalent to $P$, there exists a projective transformation of $\R^d$ that realizes the given combinatorial isomorphism from $P$ to~${P}'$.

\begin{rem*} Theorem~\ref{thm:ratprj} is sharp: We cannot hope that every polytope is the face of a projectively unique polytope. Indeed, it is a consequence of the Tarski-Seidenberg Theorem \cite{BierstoneMilman, Lindstrom} that every combinatorial type of polytope has an algebraic realization. In particular, every projectively unique polytope, and every single one of its faces, must be projectively equivalent to an algebraic polytope. Hence, a $d$-dimensional polytope with $n\ge d+3$ vertices whose set of vertex coordinates consists of algebraically independent transcendental numbers is not a face of any projectively unique polytope.
\end{rem*}

\begin{rem*}
A consequence of Theorem~\ref{thm:ratprj} is that for every finite field extension $F$ over $\mathbb{Q}$, there exists combinatorial type of polytope $\mathrm{NR}(F)$ that is projectively unique, but not realizable in any vector space over $F$. This extends on a famous result of Perles, who constructed a projectively unique polytope that is not realizable in any rational vector space, cf.\ \cite[Sec.\ 5.5, Thm.\ 4]{Grunbaum}.
\end{rem*}

In the second part of this paper, we consider a conjecture of Shephard, who asked whether every polytope is a \Defn{subpolytope} of some stacked polytope, i.e.\ whether it can be obtained as the convex hull of some subset of the vertices of some stacked polytope. While he proved this wrong in~\cite{Shephard74}, he conjectured it to be true in a combinatorial sense.

\begin{conjecture*}[Shephard~\cite{Shephard74}, Kalai {\cite[p.\ 468]{Kalai}}, \cite{KalaiKyoto}]\label{con:shka}
For every $d\ge 0$, every combinatorial type of $d$-dimensional polytope can be realized using subpolytopes of $d$-dimensional stacked polytopes.
\end{conjecture*}

The conjecture is true for $3$-dimensional polytopes, as seen by K\"omhoff in~\cite{Komhoff80}, but remained open for dimensions $d> 3$. On the other hand, Theorem~\ref{thm:ratprj} encourages us to attempt a disproof of Shephard's conjecture. The idea is to use the universality theorem above to provide a projectively unique polytope that is not a subpolytope of any stacked polytope. Since any admissible projective transformation of a stacked polytope is a stacked polytope, no realization of the polytope provided this way is a subpolytope of any stacked polytope. 

Unfortunately, the method of Theorem~\ref{thm:ratprj} is highly ineffective: The counterexample to Shephard's conjecture it yields is of a very high dimension. We use a refined method, building on the same idea, to present the following result.

\begin{thmmain}\label{thm:proj}
There exists a combinatorial type of $5$-dimensional polytope that cannot be realized as a subpolytope of any stacked polytope.
\end{thmmain}

It remains open to decide whether every combinatorial type of $4$-dimensional polytope can be realized as the subpolytope of some stacked polytope. 

\setcounter{subsection}{20}
\subsection{Universality of projectively unique polytopes}
\subsubsection*{Point configurations, PP configurations and weak projective triples}
We recall the basic facts about projectively unique point configurations and polytope-point configurations, compare also \cite[Sec.\ 5.1 \& 5.2]{AZ12}, \cite[Sec.\ 4.8 Ex.\ 30]{Grunbaum} or \cite[Pt.\ I]{RG}.

\begin{definition}[PP configurations, Lawrence equivalence, projective uniqueness]
A \Defn{point configuration} is a finite collection $R$ of distinct points in~$\R^d$. If $H$ is an oriented hyperplane in $\R^d$, then we use $H_+$ resp.\ $H_-$ to denote the open halfspaces 
bounded by $H$. If $P$ is a polytope in $\R^d$ such that $P\cap R=\emptyset$ then the pair $(P,R)$ is a 
\Defn{polytope--point configuration}, or short \Defn{PP configuration}. 

A hyperplane $H$ is \Defn{external} to $P$ if $H\cap P$ is a face of~$P$. Two PP configurations $(P,R)$, $(P',R')$ in $\R^d$ are \Defn{Lawrence equivalent} if there is a 
bijection $\varphi$ between the vertex sets of $P$ and $P'$ and the sets $R$ and~$R'$, 
such that, if $H$ is any hyperplane  for which the closure of $H_-$ contains $P$, there exists an oriented hyperplane $H'$ for which the closure of $H'_-$ contains $P'$ and  
\[
    \varphi(\F_0(P)\cap H_-)=\F_0(P')\cap H'_-, \qquad 
    \varphi(R\cap H_+)={R'}\cap H'_+, \qquad 
    \varphi(R\cap H_-)={R'}\cap H'_-;
\]
where $\F_0(P)$ denotes the set of vertices of $P$.

A PP configuration $(P,R)$ in~$\R^d$ is \Defn{projectively unique} if for any PP configuration $(P',R')$ in~$\R^d$ 
Lawrence equivalent to it, and every bijection $\varphi$ that induces the Lawrence equivalence, there is a projective transformation $T$ that realizes $\varphi$. A point configuration $R$ is \Defn{projectively unique} if the PP configuration $(\emptyset, R)$ is projectively unique, and it is not hard to verify that a polytope $P$ is projectively unique if and only if the PP configuration $(P,\emptyset)$ is projectively unique.
\end{definition}

\begin{prp}[Lawrence extensions, cf.\ {\cite[Prp.\ 5.2]{AZ12}}, {\cite[Lem.\ 3.3.3 \& 3.3.5]{RG}}]\label{prp:mlwextn}
Let $(P,R)$ be a projectively unique PP configuration in $\R^d$. Then there exists a $(\dim P+\#R)$-dimensional polytope on $f_0(P) + 2\cdot \#R$ vertices that is projectively unique and that contains $P$ as a face.
\end{prp}
Here $\#R$ denotes the cardinality of $R$.
\begin{definition}[Framed PP configurations]\label{def:framep} Let $(P, R)$ denote any PP configuration in $\R^d$, and let $Q$ be any subset of $\F_0(P) \cup R$.
Let $(P',R')$ be any PP configuration in $\R^d$ Lawrence equivalent to $(P, R)$, and let $\varphi$ denote the labeled isomorphism inducing the Lawrence equivalence. 

The PP configuration $(P, R)$ is \Defn{framed} by the set $Q$ if $\varphi_{|Q}=\mathrm{id}_{|Q}$ implies $\varphi=\mathrm{id}$. Similarly, a polytope~$P$ (resp.\ a point configuration~$R$) is \Defn{framed} by a set $Q$ if $(P,\emptyset)$ (resp.\ $(\emptyset,R)$) is framed by $Q$.
\end{definition}

\begin{examples}[Some instances of framed PP configurations]\label{ex:stdet} $\mbox{}$
\begin{compactenum}[(i)]
\item If $(P,R)$ is any PP configuration, then $\F_0(P)\cup R$ frames $(P,R)$.
\item If $(P,R)$ is any PP configuration framed by a set $Q$, then every superset of $Q$ frames $(P,R)$ as well.
\item If $(P,R)$ is any projectively unique PP configuration, and $Q\subset \F_0(P) \cup R$ is a projective basis, then $Q$ frames $(P,R)$.
\item Any $d$-cube, $d\geq 3$, is framed by $2^d-1$ of its vertices, cf.\ \cite[Lem.~3.4]{AZ12}.
\end{compactenum}
\end{examples}

\begin{definition}[Weak projective triple in $\R^d$]\label{def:wpt}
A triple $(P,Q,R)$ of a polytope $P$ in $\R^d$, a subset $Q$ of $\F_0(P)$ and a point configuration $R$ in $\R^d$ is a \Defn{weak projective triple} in $\R^d$ if and only if

\begin{compactenum}[(1)]
\item $(\emptyset, Q \cup R)$ is a projectively unique point configuration,
\item $Q$ frames the polytope $P$, and
\item some subset of $R$ spans a hyperplane $H$, the \Defn{wedge hyperplane}, which does not intersect $P$.
\end{compactenum}
\end{definition}

\begin{figure}[htbf]
\centering 
  \includegraphics[width=0.82\linewidth]{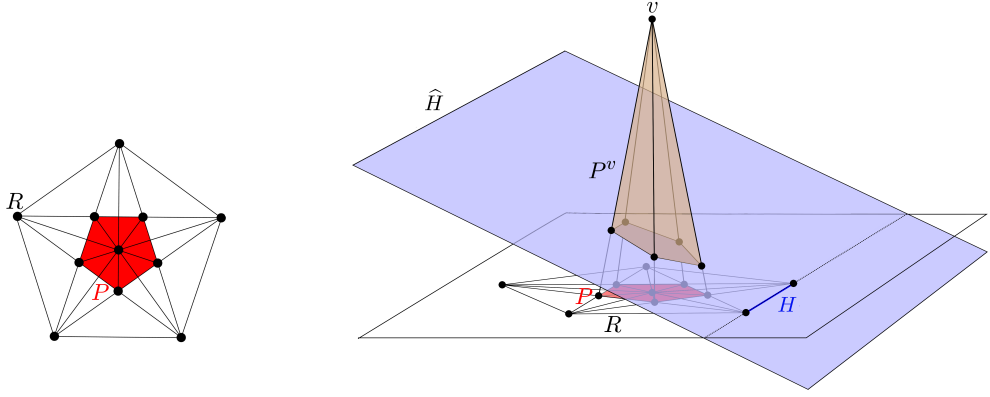} 
  \caption{The subdirect cone of the weak projective triple $(P,Q,R)$ in the special case where $Q$ coincides with the vertex set of $P$.} 
  \label{fig:subdirect}
\end{figure}

\begin{definition}[Subdirect cone]\label{def:subd}
Let $(P,Q,R)$ be a weak projective triple in $\R^d$, seen as a canonical subspace of $\R^{d+1}$. Let $H$ denote the wedge hyperplane in $\R^d$ spanned by points of $R$ with $H\cap P=\emptyset$. Let $\vect v$ denote any point not in $\R^d$, and let $\widehat{H}$ denote any hyperplane in $\R^{d+1}$ such that $\widehat{H}\cap\R^d=H$ and $\widehat{H}$ separates $\vect v$ from $P$. Consider, for every vertex $\vect p$ of $P$, the point $\vect{p}^{\vect v}=\conv\{\vect v,\,\vect p\}\cap \widehat{H}$. Denote by $P^{\vect v}$ the pyramid
\[P^{\vect v}:=\conv \Big( \vect v\cup \bigcup_{\vect p\in \F_0(P)} \vect{p}^{\vect v} \Big).\]
 The PP configuration $(P^{\vect {v}},Q \cup R)$ in $\R^{d+1}$ is a \Defn{subdirect cone} of $(P,Q,R)$.
\end{definition}

\begin{lemma}[{\cite[Lem.~5.8.]{AZ12}}]\label{lem:subdc}
For any weak projective triple $(P,Q,R)$ the subdirect cone $(P^{\vect {v}},Q \cup R)$ is a projectively unique PP configuration, and the base of the pyramid $P^{\vect {v}}$ is projectively equivalent to $P$.
\end{lemma}	

Combining Lemma~\ref{lem:subdc} and Proposition~\ref{prp:mlwextn} gives:

\begin{cor}\label{cor:wpt}
If $(P,Q,R)$ is a weak projective triple, there exists a projectively unique polytope of dimension $\dim P + \#Q + \#R +1$ on $f_0(P)+2\cdot(\#Q + \#R)+1$ vertices that contains a face projectively equivalent to $P$. \qed
\end{cor}



\subsubsection*{Constructions for projectively unique point configurations and the proof of Theorem~\ref{thm:ratprj}}

Let $P$ be any algebraic polytope. Our goal for this section is to find a weak projective triple that contains~$P$. Applying Corollary~\ref{cor:wpt} then finishes the proof of Theorem~\ref{thm:ratprj}. The main step towards that goal is to embed $\F_0(P)$ into a projectively unique point configuration. In the construction, we will use the following straightforward observation repeatedly.

\begin{lemma}\label{lem:union}
Let $R$ be a projectively unique point configuration, let $Q\subseteq R$ and let $R'\supseteq Q$ be a point configuration framed by $Q$. Then $R\cup R'$ is a projectively unique point configuration.\qed
\end{lemma}


\begin{prp}\label{prp:Qdm}
The point configuration $\Q^d:=\{v\in \Z^d\subset \R^d : ||v||_\infty\le  1\}$
is projectively unique for every $d\ge3$.
\end{prp}
\begin{proof}
The proof is by induction on $d$. We start proving that $\Q^3$ is projectively unique. This implies that $\Q^d$ is projectively unique for any $d\geq 3$.

\medskip
\noindent \textbf{$\Q^3$ is projectively unique:} To see that $\Q^3$ is projectively unique, we start with the folklore observation that the points $(\pm 1, \pm 1, \pm 1)$, together with the origin $(0,0,0)$, form a projectively unique configuration $W\subsetneq \Q^3$ (cf.\ Figure~\ref{fig:Q31_1}). Furthermore, we claim $W$ frames $\Q^3$, thereby proving that $\Q^3$ is projectively unique since $W$ is projectively unique.

\begin{figure}[htpb]
\centering
\begin{subfigure}[t]{.3\linewidth}
\centering
\quad\includegraphics[width=.7\linewidth]{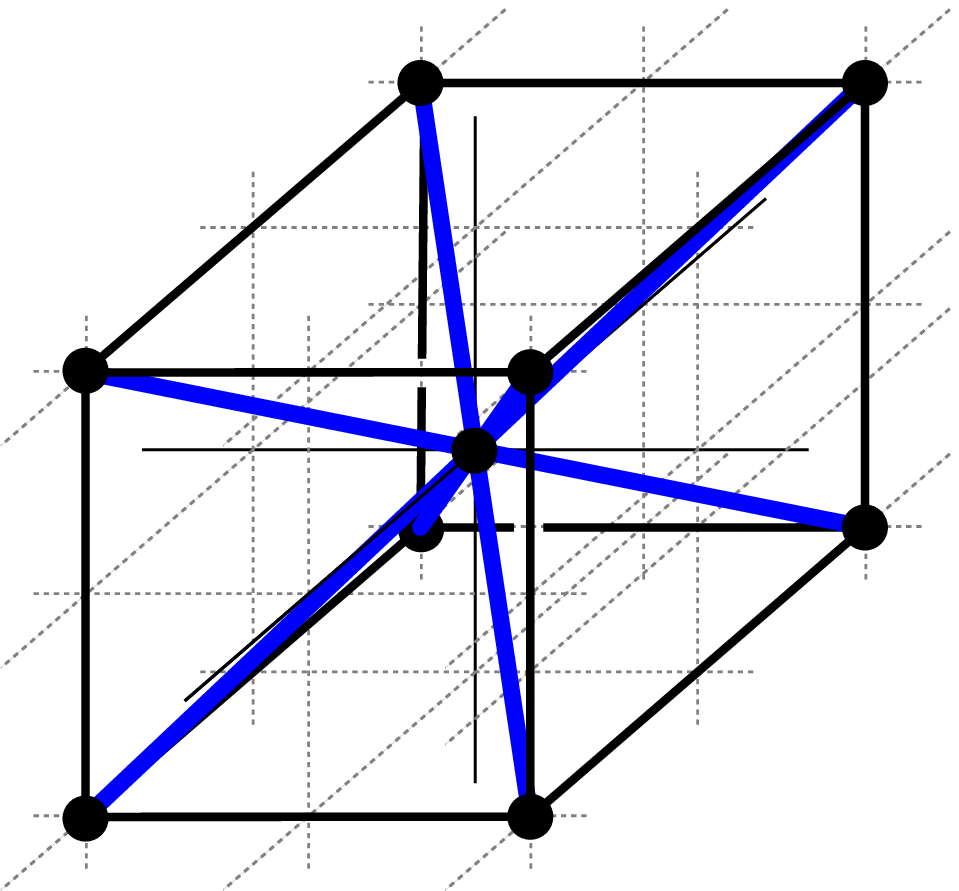}
\caption{$W$}\label{fig:Q31_1}
\end{subfigure}\qquad
\begin{subfigure}[t]{.3\linewidth}
\centering
\includegraphics[width=.7\linewidth]{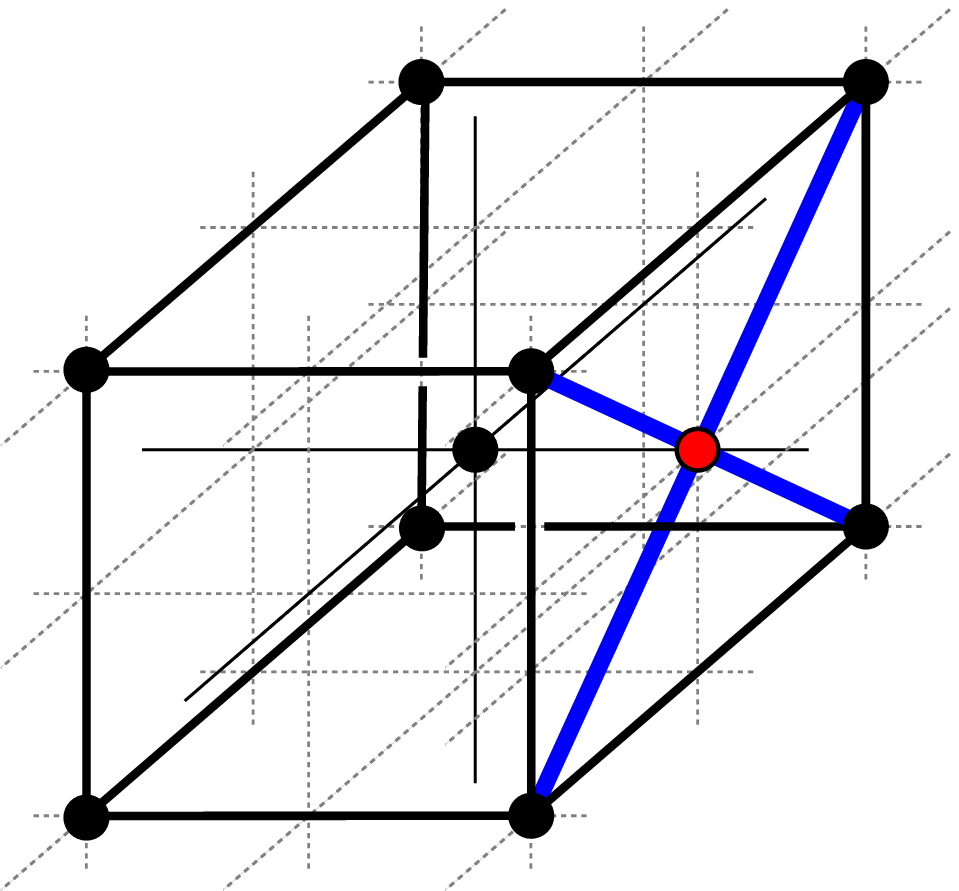}
\caption{Center points of facets.}\label{fig:Q31_2}
\end{subfigure}\qquad
\begin{subfigure}[t]{.3\linewidth}
\centering
\includegraphics[width=.7\linewidth]{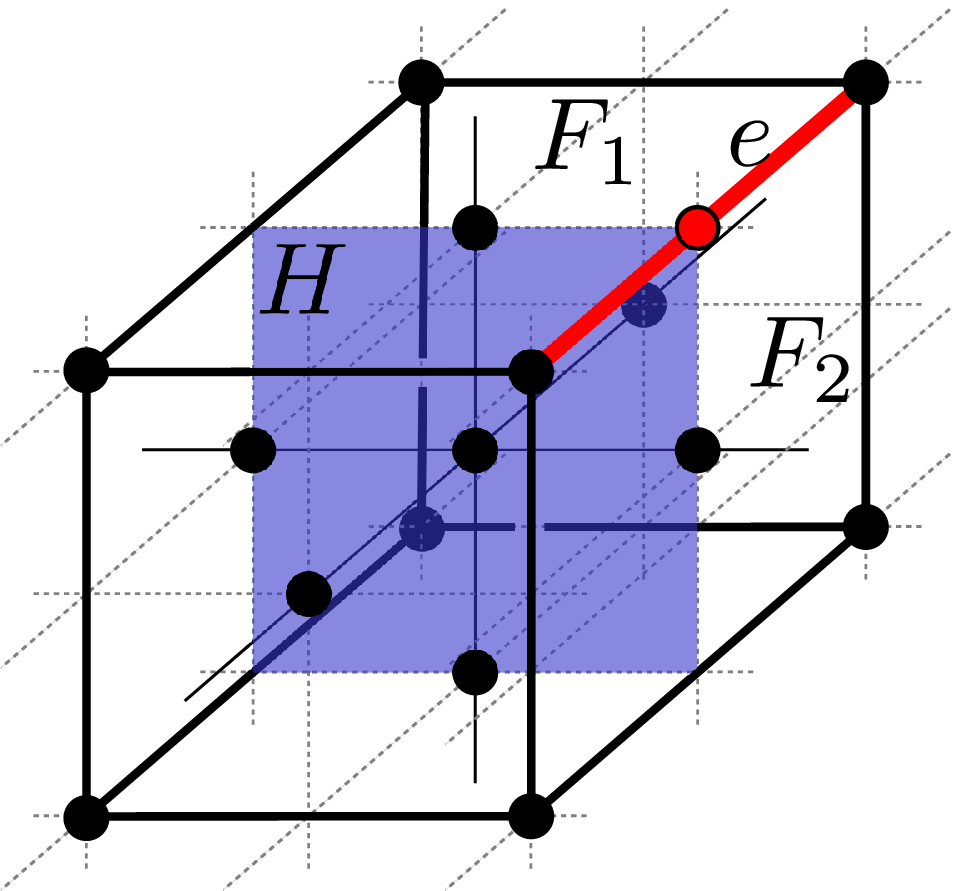}
\caption{Center points of edges.}\label{fig:Q31_3}
\end{subfigure}
\caption{Showing that $\Q^3$ is projectively unique.}\label{fig:Q31}
\end{figure}

To see this, notice that the point $(1,0,0)$ of $\Q^3\setminus W$ is determined as the intersection of the lines $\aff\{(+1,+1,+1),\, (+1,-1,-1)\}$ and $\aff\{(+1,+1,-1),\, (+1,-1,+1)\}$, which are spanned by points of~$W$. Similarly, all points that arise as coordinate permutations and/or sign changes from $(+1,0,0)$ are determined this way. Geometrically, these are the center points of the facets of the cube $[-1,1]^3=\conv W=\conv \Q^3$ (cf.~Figure~\ref{fig:Q31_2}). 

The remaining lattice points of $\Q^3$ coincide with the midpoints of the edges of said cube. To determine them, let $e$ be any edge of $\Q^3$ and let $F_1$ and $F_2$ be the facets of $\Q^3$ incident to that edge. Finally, let $H$ be the hyperplane spanned by the center point of $\Q^3$ and the center points of $F_1$ and $F_2$. The midpoint of~$e$ is the unique point of intersection of $e$ and $H$ (cf.\ Figure~\ref{fig:Q31_3}).

\medskip
\noindent \textbf{$\Q^d$ is projectively unique:} For $d\geq 4$, consider the projective basis $B$ of $\R^d$ consisting of the vertex $v_0:=(+1,+1,\dots,+1)$ of $[-1,1]^d$, together with the neighboring vertices $v_1:=(-1,+1,\dots,+1)$, $v_2:=(+1,-1,\dots,+1)$, $\dots$, $v_d:=(+1,+1,\dots,-1)$ and the origin $o:=(0,\dots,0)$ (cf.\ Figure~\ref{fig:Qd1_1}). We will see that once the coordinates of the elements in $B$ are fixed, then the coordinates of all the remaining lattice points of $\Q^d$ can be determined uniquely.

\begin{figure}[htpb]
\centering
\begin{subfigure}[t]{.3\linewidth}
\centering
\includegraphics[width=.7\linewidth]{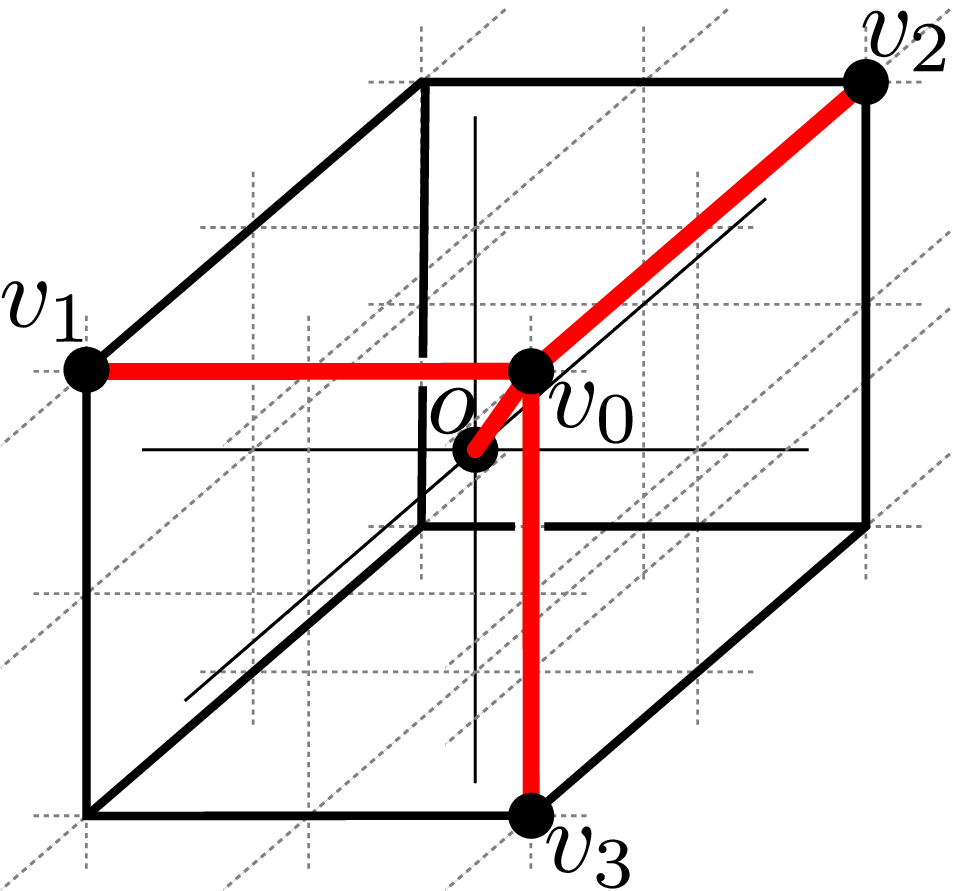}
\caption{$B$}\label{fig:Qd1_1}
\end{subfigure}\qquad
\begin{subfigure}[t]{.3\linewidth}
\centering
\includegraphics[width=.7\linewidth]{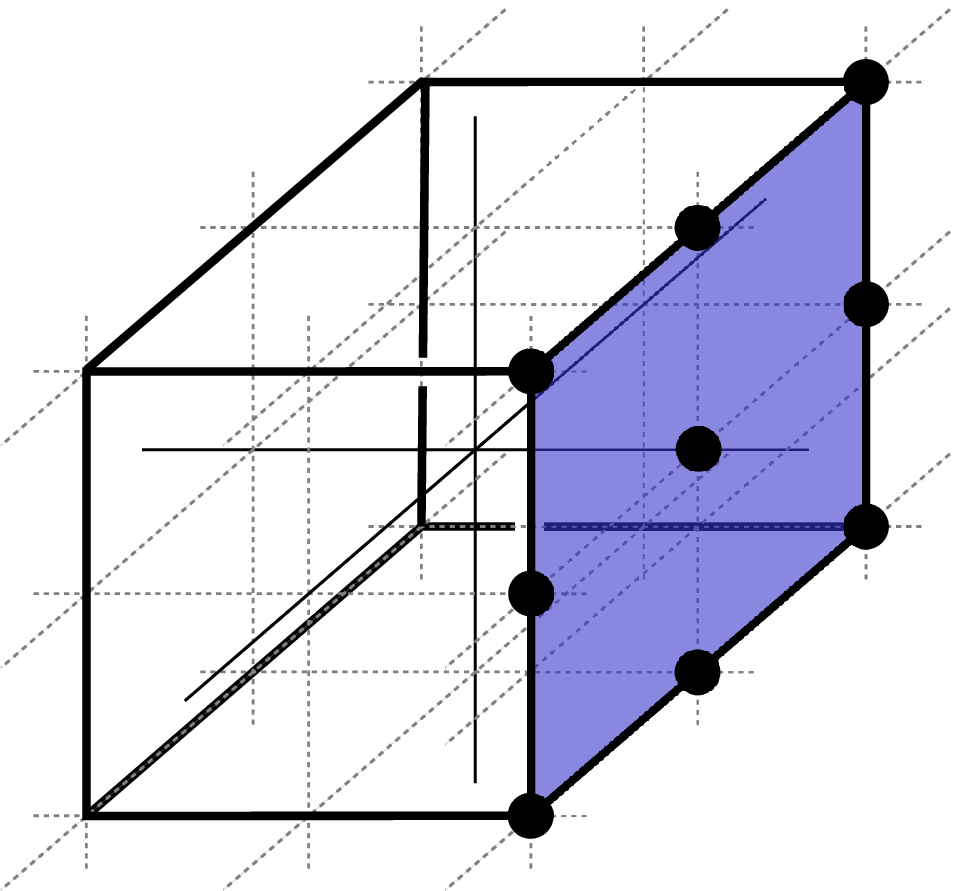}
\caption{$R_1$}\label{fig:Qd1_2}
\end{subfigure}\qquad
\begin{subfigure}[t]{.3\linewidth}
\centering
\includegraphics[width=.7\linewidth]{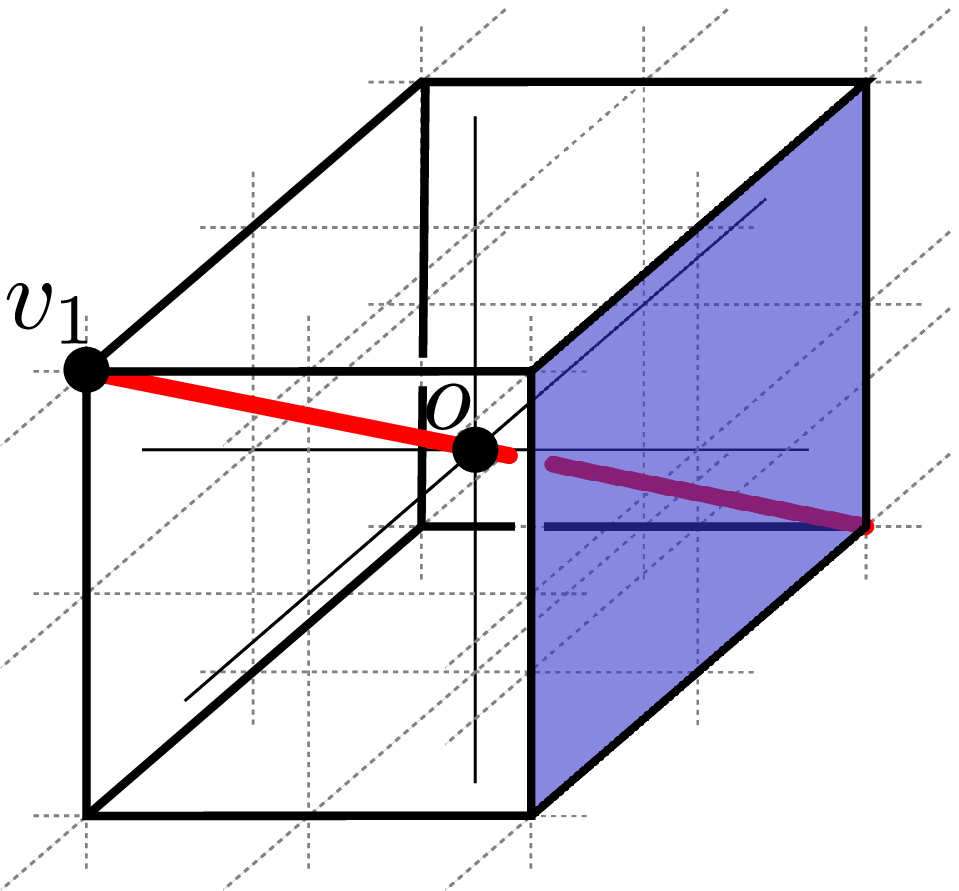}
\caption{$\aff\{o,v_1\}\cap\aff\{R_1\}$}\label{fig:Qd1_3}
\end{subfigure}\qquad
\caption{Scheme for showing that $\Q^d$ is projectively unique. (The picture displays $\Q^3$ for the sake of clarity, but the proof starts with $d\geq4$.)}\label{fig:Qd1}
\end{figure}

Consider the set of points of $\Q^d$ lying in a common facet of $[-1,1]^d$ that is incident to $v_0$; for example, $R_1:=\Q^d\cap\aff\{v_0,v_2,\dots,v_d\}$ (cf.\ Figure~\ref{fig:Qd1_2}). Observe that $R_1$ is just an affine embedding of $\Q^{d-1}$ into~$\R^d$. As such, $R_1$ is projectively unique, and thus it is determined uniquely if a projective basis for its span is fixed.

Clearly, the points $v_0,v_2,\dots,v_d$ of $B$ form an affine basis for the affine span of $R_1$. Furthermore, the coordinates of the point $w=(+1,-1,\dots,-1)$, are fixed by $B$. Indeed, $w$ is the the point of intersection of the line $\aff\{o,v_1\}$ with the hyperplane $\aff\{R_1\}$ (cf.\ Figure~\ref{fig:Qd1_3}). To sum up, we have that
\begin{compactitem}[$\circ$]
\item the points $v_0,v_2,\dots,v_d, w$ are determined uniquely from the points of $B$,
\item the points $v_0,v_2,\dots,v_d, w$ are elements of $R_1$, and
\item the points $v_0,v_2,\dots,v_d, w$ form a projective basis for the span of $R_1$. 
\end{compactitem}
Consequently, $R_1\cup B$ is a projectively unique point configuration, since $B$ is projectively unique. We can repeat this argumentation for all point configurations \[R_i:=\Q^d\cap\aff\left(\{v_0,v_1, v_2,\dots,v_d\} \setminus \{v_i\}\right),\ i\in \{1,\dots, d\}.\] In particular, \[\widetilde{\Q}^d=B\cup\bigcup_{i\in \{1,\dots, d\}} R_i\]
is projectively unique. Moreover, since the last vertex of a cube of dimension $d\geq 3$ is determined by the remaining ones by (cf.\ \cite[Lem~3.4]{AZ12}, compare also Example~\ref{ex:stdet}(iv)), the configuration $\widetilde{\Q}^d\cup \{-v_0\}$ is projectively unique as well. By symmetry,

\[-\widetilde{\Q}^d\cup \{v_0\}=-B\cup\bigcup_{i\in \{1,\dots, d\}} -R_i \cup \{v_0\}\]

\noindent is also projectively unique.
Clearly, $\widetilde{\Q}^d\cup \{-v_0\}$ and $-\widetilde{\Q}^d\cup \{v_0\}$ intersect along a projective basis: for instance, the set $B$ lies in both $-\widetilde{\Q}^d\cup \{v_0\}$ and $\widetilde{\Q}^d\cup \{-v_0\}$ and forms a projective basis as desired. Thus, the point configuration $\Q^d=\widetilde{\Q}^d\cup \{-v_0\} \cup -\widetilde{\Q}^d\cup \{v_0\}$ is projectively unique.
\end{proof}

\paragraph*{Embedding vertex sets of algebraic polytopes}

We start with a point configuration $\PROJ{\vect p}$ that shows that it is enough to fix each coordinate of a point to frame it. 

\begin{lemma}\label{lem:cube}
For each point $\vect p$ in the positive orthant $\R^d_+$ of $\R^d,\ d\ge 3$, there is a point configuration $\PROJ{\vect p}\in \R^d$ that contains $\vect p=(p_1,\dots,p_d)$ and is framed by the points \[L(\vect p):=\left\{\veczero,p_1 \ve_1,\dots,p_d \ve_d, \tfrac{p_1\ve_1}{2} ,\dots,\tfrac{p_d\ve_d}{2} \right\}.\]
\end{lemma}

\begin{proof}
We denote by $\Q^d+\mathbf{1}$ the translation of $\Q^d$ by the all-ones vector. Moreover, let $\mathrm{D}=\mathrm{D}[p_1,\dots,p_d]$ denote the diagonal matrix with diagonal entries $p_1,\dots,p_d$. Notice that $\PROJ{\vect p}:=\frac{\mathrm{D}}{2}(\Q^d+\mathbf{1})$ is projectively unique (by Proposition~\ref{prp:Qdm}), contains $\vect p$ and the set $L(\vect p)$. Since every projective transformation fixing $L(\vect p)$ is the identity, the subset $L(\vect p)$ frames $\PROJ{\vect p}$.
\end{proof}

Finally, we only need to frame each coordinate of the point. The idea is to realize the defining polynomial of any real algebraic number in a \Defn{functional arrangement} (cf.\ \cite[Def.\ 9.6]{Kapovich}), which conversely determines the real algebraic number. 

\begin{definition}\label{def:func}
For a function $f:\R^k\mapsto \R$, a \Defn{functional arrangement} $\FUNC{f}=\FUNC{f}(\vect x)$ for $f$ is a $k$-parameter family of point configurations in $\R^2$ such that the following conditions hold:
\begin{compactenum}[\rm (i)]
\item For all $\vect x=(x_1,\dots, x_k)$ in $\R^k$, the functional arrangement $\FUNC{f}(\vect x)$ contains the \Defn{output point} $f(\vect x) e_1$, the \Defn{input points} $x_i e_1,\ i\in\{1,\dots,k\},$ and the set $\Q^2+\mathbf{1}$. 
\item For all $\vect x\in \R^k$, the set $\{x_i e_1: i\in\{1,\dots,k\}\} \cup (\Q^2+\mathbf{1})$ frames the point configuration $\FUNC{f}(\vect x)$.
\end{compactenum}
For the last condition, let $\varphi(\FUNC{f})(x)$ denote any point configuration Lawrence equivalent the functional arrangement $\FUNC{f}(x)$, where $\varphi$ is the bijection of points that induces the Lawrence equivalence.
\begin{compactenum}[\rm (i)]
\setcounter{enumi}{2}
\item For all $\vect x\in \R^k$, if $\varphi$ is the identity on $\Q^2+\mathbf{1}$ and $\varphi(f(x)e_1)=f(x)e_1$, we have $\varphi(xe_1)\in f^{-1}(f(x))e_1 \subset \R^{2\times k}$.
\end{compactenum}
\end{definition}

Hence, a functional arrangement essentially computes a function and its inverse by means of its point-line incidences alone. An just as like functions, they can be combined:

\begin{lemma}\label{lem:composition}
Let $\mathrm{F}(x,z)$ and $\mathrm{G}(y,z)$, $x\in \R^k, y\in \R^\ell, z\in \R^m$, denote functional arrangements for functions $f:\R^{k+m}\mapsto \R$ and $g:\R^{\ell+m}\mapsto \R$, respectively. Then \[\mathrm{F}(g(y,z),x',z)\cup \mathrm{G}(y,z),\quad x':=(x_2,\dots,x_k)\]
is a functional arrangement for the function $f(g(y,z),x',z)$ from  $\R^{k+m+\ell-1}$ to $\R$.\qed
\end{lemma}

\begin{prp}[cf.\ \cite{Staudt}, {\cite[Thm.~D]{Kapovich}}]\label{prp:VonStaudt}
Every integer coefficient polynomial $\psi$ is realized by a functional arrangement $\FUNC{\psi}$.  
\end{prp}

\begin{proof}
The proof of this fact is based on the classical von Staudt constructions (\cite{Staudt}, compare also \cite[Ch.~5]{Richter-Gebert2011}, \cite[Sec.~11.7]{RG} or \cite[Sec.~5]{Kapovich}), which are a standard tool to encode algebraic operations in point-line incidences. 

To construct the desired functional arrangements, notice that every integer coefficient polynomial in variable $x$ can be written using $0$, $1$ and $x$, combined by addition and multiplication. Hence, thanks to Lemma~\ref{lem:composition}, it suffices to provide:
\begin{compactitem}[$\circ$]
\item A functional arrangement $\mathrm{ADD}(\alpha,\beta)$ for the function $\mathrm{a}(\alpha,\beta)=\alpha+\beta$ computing the addition of two real numbers.
\item A functional arrangement $\mathrm{MLT}(\alpha,\beta)$ for the function $\mathrm{m}(\alpha,\beta)=\alpha\cdot\beta$ computing the product of two real numbers.  
\end{compactitem}
Both functional configurations are shown below. We invite the reader to check that they indeed form functional arrangements for addition and multiplication.

\begin{figure}[htpb]
\centering
\begin{subfigure}[t]{.37\linewidth}
\centering
\includegraphics[width=\linewidth]{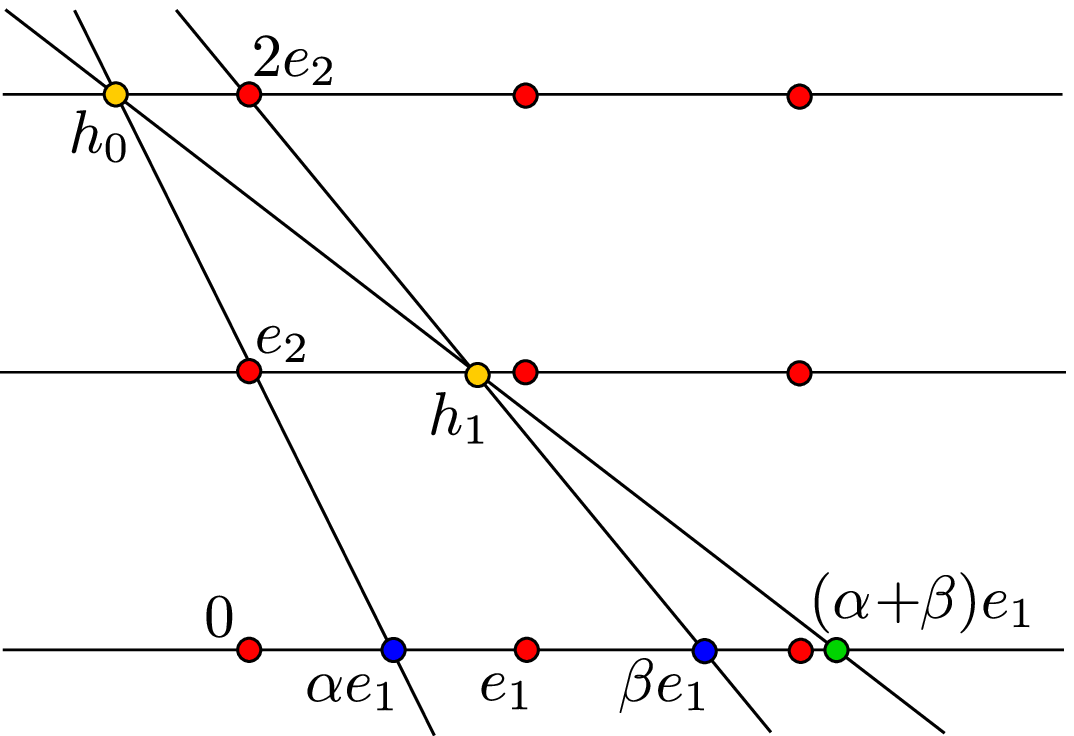}
\caption{$\mathrm{ADD}(\alpha,\beta)$}
\end{subfigure}\qquad\quad\qquad
\begin{subfigure}[t]{.37\linewidth}
\centering
\includegraphics[width=\linewidth]{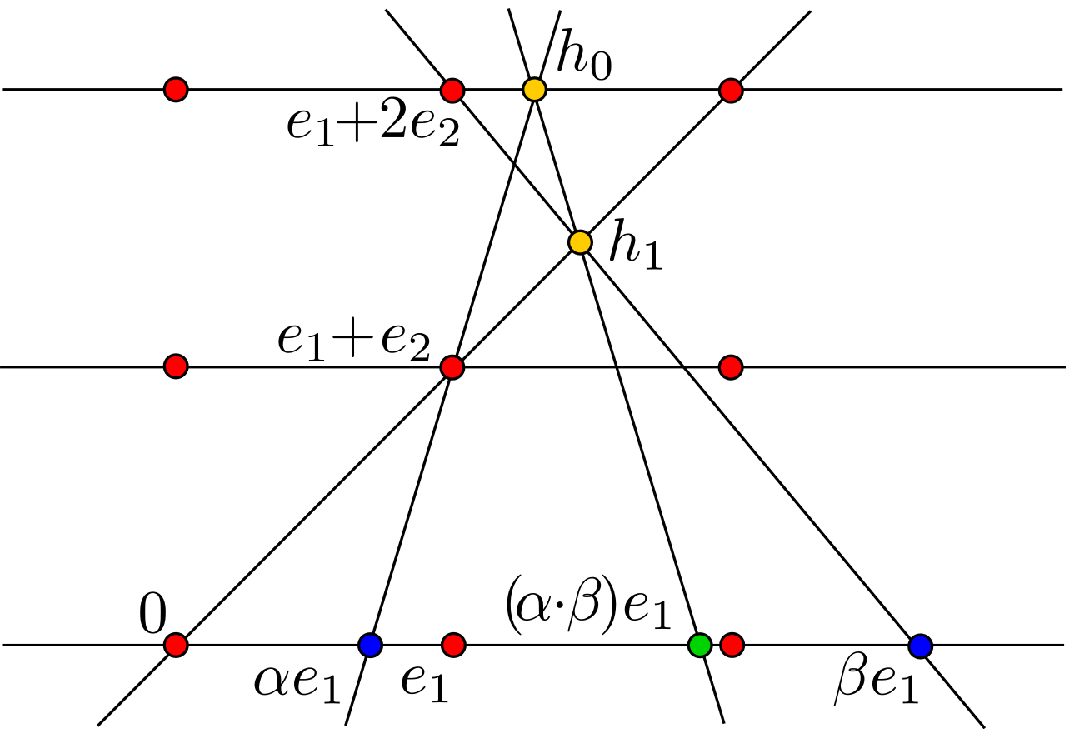}
\caption{$\mathrm{MLT}(\alpha,\beta)$ }
\end{subfigure}
\caption{Von Staudt constructions for addition, $\mathrm{ADD}(\alpha,\beta)$, and multiplication, $\mathrm{MLT}(\alpha,\beta)$. The blue points of the configuration form the input, the red points show $\Q^2+\mathbf{1}$, the yellow points are auxiliary to the construction and the green points give the output.}  \label{fig:VonStaudt}
\end{figure}

By switching output and input points of these functional arrangements, we also obtain functional arrangements $\mathrm{SUB}(\alpha,\beta)$ and $\mathrm{DIV}(\alpha,\beta)$ for $\mathrm{s}(\alpha,\beta)=\alpha-\beta$ and  $\mathrm{d}(\alpha,\beta)=\frac{\alpha}{\beta}$.
\end{proof}

\begin{example}
Let us construct a functional arrangement for $x\mapsto x^2-2=\mathrm{s}(\mathrm{m}(x,x),\mathrm{a}(1,1))$. Using Lemma~\ref{lem:composition}, this arrangement can be written as combination of the functional arrangements for addition, subtraction and multiplication:
\begin{align*}
\FUNC{x^2-2}(x)=\FUNC{\mathrm{s}(\mathrm{m}(x_1,x_2),\mathrm{a}(x_3,x_4))}(x,x,1,1)= \mathrm{SUB}(x^2,2) \cup \mathrm{MLT}(x,x) \cup \mathrm{ADD}(1,1).
\end{align*}
Figure~\ref{fig:VonStaudtExample} shows the evaluations of this functional arrangement at $\sqrt{2}$ and $\sqrt{3}$.

\begin{figure}[h!tpf]
\centering
\begin{subfigure}[t]{.47\linewidth}
\centering
\includegraphics[width=\linewidth]{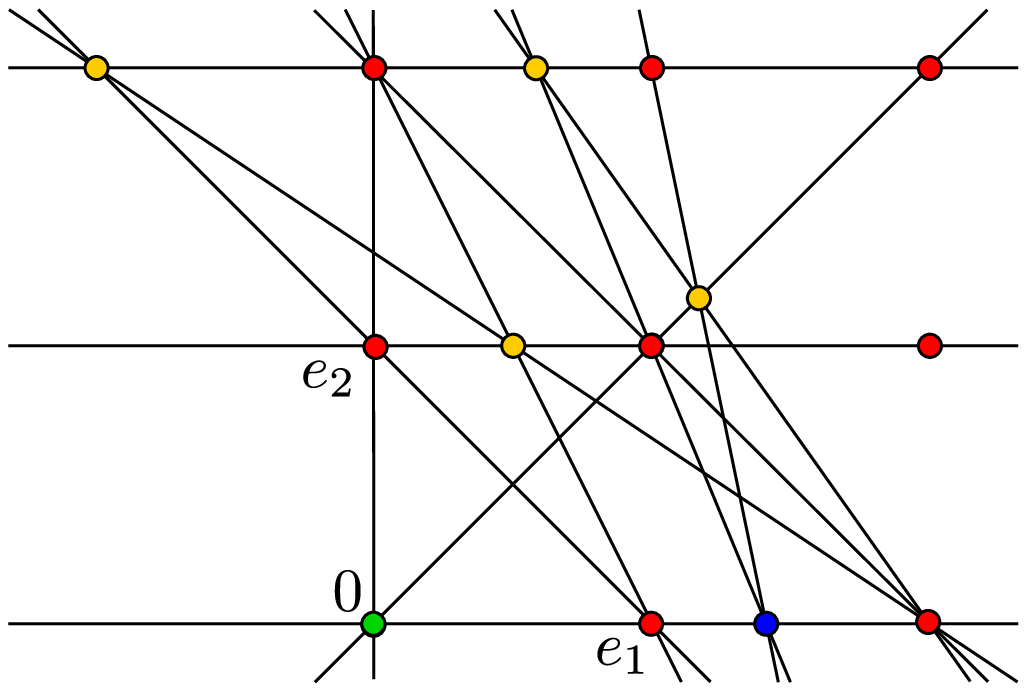}
\caption{$\FUNC{x^2-2}(\sqrt{2})$}\label{sfig:sqrt2}
\end{subfigure}\qquad
\begin{subfigure}[t]{.47\linewidth}
\centering
\includegraphics[width=\linewidth]{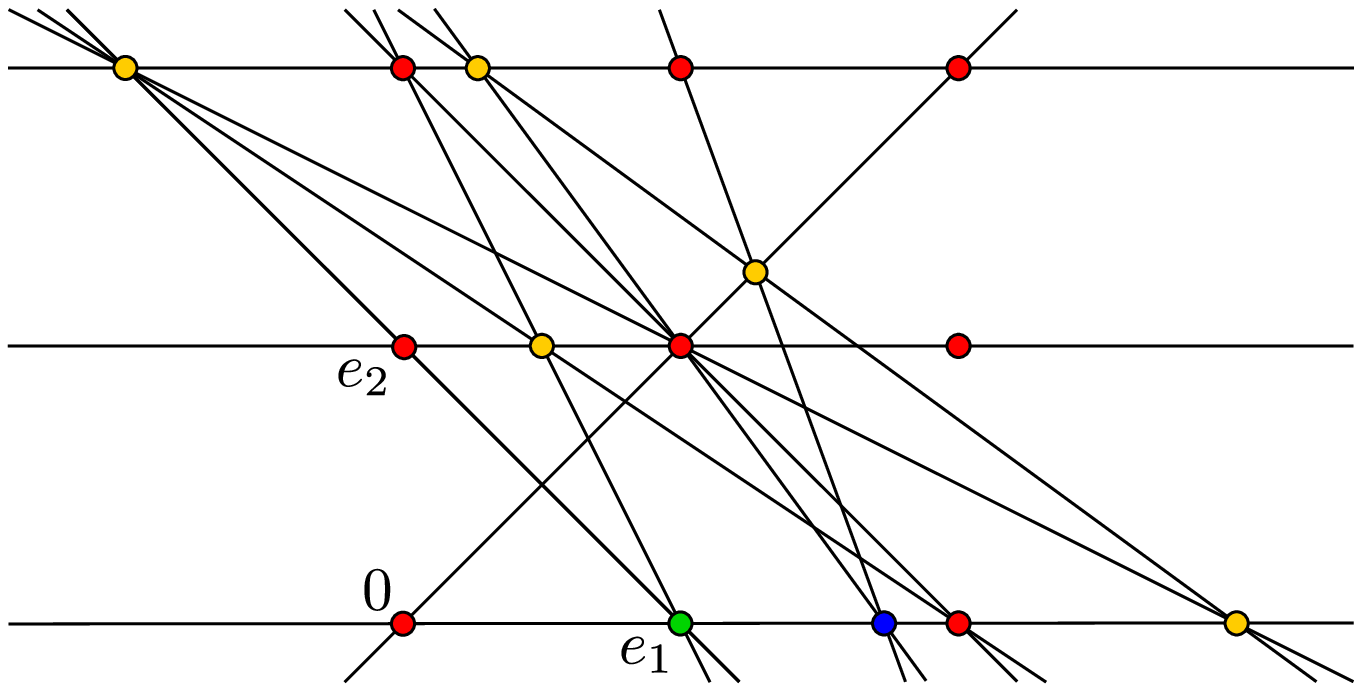}
\caption{$\FUNC{x^2-2}(\sqrt{3})$}
\end{subfigure}
\caption{Two evaluations of the functional arrangement $\FUNC{x^2-2}$.}\label{fig:VonStaudtExample}
\end{figure}

The point configuration $\FUNC{x^2-2}(\sqrt{2})$, as given above, is framed by $\Q^2+\mathbf{1}$. Hence, it enables us to compute $\sqrt{2}$ from $\Q^2+\mathbf{1}$. Similar point configurations for any algebraic number are given in the following corollary.

\end{example}

\begin{cor}\label{cor:vonS}
For each real algebraic number $\zeta$, there is a point configuration $\COOR{\zeta}\subset \R^2$ framed by $\Q^2+\mathbf{1}$ such that $\zeta \ve_1\in \COOR{\zeta}$.
\end{cor}

\begin{proof}
If $\zeta$ is algebraic of degree $\le 1$ (i.e.\ $\zeta$ is rational), let $p,q\in\Z$ be such that $\zeta=\frac{p}{q}$ and set $\psi(x)=qx-p$. We then define $\COOR{\zeta}:=\FUNC{\psi}(\zeta)$.

If $\zeta$ is instead of degree $\ge 2$ (i.e.\ $\zeta$ is irrational, but algebraic), and $\psi$ is an integer coefficient polynomial with root $\zeta$, then let $\zeta^+$ and $\zeta^-$ denote rational numbers with the property that $\zeta$ is the only root of $\psi$ in the interval $[\zeta^-, \zeta^+]$. 
Then the desired point configuration is given by \[\COOR{\zeta}:=\COOR{\zeta^-} \cup \FUNC{\psi}(\zeta)\cup \COOR{\zeta^+}, \]
which contains $\zeta$ by construction, and it is framed by $\Q^2+\mathbf{1}$. Indeed, let $\varphi(\COOR{\zeta})$ be a configuration Lawrence equivalent to $\COOR{\zeta}$, where the equivalence is induced by the bijection $\varphi$. If $\varphi$ is the identity on $\Q^2+\mathbf{1}$, then by Definition~\ref{def:func}(iii) and since ${\psi}(\zeta)e_1=0\in\Q^2+\mathbf{1}$, we obtain $\varphi(\zeta^-e_1)=\zeta^-e_1$, $\varphi(\zeta^+e_1)=\zeta^+e_1$ and $\varphi(\zeta e_1)=\zeta' e_1$. Here $\zeta'$ must be a root of $\psi$, which lies in the interval $[\zeta^-, \zeta^+]$ by Lawrence equivalence. This root is unique, so $\varphi(\zeta e_1)=\zeta e_1$. 
To sum up, $\Q^2+\mathbf{1}$ determines $\zeta e_1$, which together with $\Q^2+\mathbf{1}$ frames $\COOR{\zeta}$ by Definition~\ref{def:func}(ii). 
\end{proof}

\begin{cor}\label{cor:algpoint}
Let $\zeta$ be any point in $\R^d_+$, $d\ge 3$, with algebraic coordinates. Then there is a projectively unique point configuration $\COOR{\zeta}$ containing $\zeta$ and $\Q^d+\mathbf{1}$.
\end{cor}

\begin{proof}
Let  $\zeta=(\zeta_1, \dots, \zeta_d)$. For any $i$, let $\COOR{\zeta_i}_{i,i+1}$ resp.\ $\COOR{\nicefrac{\zeta_i}{2}}_{i,i+1}$ denote the configurations provided by Corollary~\ref{cor:vonS}, naturally embedded into the plane spanned by $e_i$ and $e_{i+1}$ (using a cyclic labelling for the base vectors $e_{\ast}$). We obtain in \[\COOR{\zeta}:=(\Q^d+\mathbf{1})\cup \PROJ{\zeta}\cup\bigcup_{i\in \{1,\dots,d\}}\COOR{\zeta_i}_{i,i+1}\cup\bigcup_{i\in \{1,\dots,d\}}\COOR{\nicefrac{\zeta_i}{2}}_{i,i+1}\]
the desired point configuration: $\COOR{\zeta}$ contains $\zeta$ and $\Q^d+\mathbf{1}$ by construction and it is projectively unique by Lemma~\ref{lem:union}.
\end{proof}

\paragraph*{Conclusion of proof}

\begin{proof}[{\bf Proof of Theorem~\ref{thm:ratprj}}]
Let $P$ denote a algebraic polytope in $\R^d$. We assume that $d\ge 3$, since if $d\leq 2$ we can realize $P$ as a face of some $3$-dimensional pyramid. By dilation and translation, we may assume that $P$ lies in the interior of the cube $\mathrm{C}:=\conv (\Q^d+\mathbf{1})$. Consider now the point configuration
\[\COOR{P}:=\bigcup_{v\in \F_0(P)} \COOR{v},\] \vskip -1.5mm
\noindent where $\COOR{v},\ v\in \F_0(P),$ is the point configuration provided by Corollary~\ref{cor:algpoint}. Set $Q:=\F_0(P)$ and $R:=\COOR{P}\setminus Q$. With this, we have that $(P,Q,R)$ is a weak projective triple. Indeed,

\begin{compactenum}[(1)]
\item $Q\cup R=\COOR{P}$ is a projectively unique point configuration by Corollary~\ref{cor:algpoint} and Lemma~\ref{lem:union},
\item $Q$ obviously frames $P$ (cf.\ Example~\ref{ex:stdet}(i)), and
\item since $P\subset \mathrm{int}\, \mathrm{C}$, we have $\F_0(\mathrm{C})\subset R$, and hence any of the facet hyperplanes of $\mathrm{C}$ can be chosen as wedge hyperplane for the triple.\end{compactenum}
Thus, by Corollary~\ref{cor:wpt}, there exists a projectively unique polytope that contains a face projectively equivalent to $P$.
\end{proof}
\setcounter{subsection}{18}
\subsection{Subpolytopes of stacked polytopes}

In this section, we disprove Shephard's conjecture. While this alone could be done using Theorem~\ref{thm:ratprj} (with arguments slightly differing from those below), we here use a refined argumentation to construct combinatorial types of $5$-dimensional polytopes that are not realizable as subpolytopes of $5$-dimensional stacked polytopes (Theorem~\ref{thm:proj}). 
Instead of Theorem~\ref{thm:ratprj}, we will use the following result of Below.

\begin{theorem}[{\cite[Ch.~5]{Below2002}}, see also {\cite[Thm.~4.1]{Dobbins2011}}]\label{thm:Below}
Let $P$ be any algebraic $d$-dimensional polytope. Then there is a polytope $\widehat{P}$ of dimension $d+2$ that contains a face $F$ that is projectively equivalent to $P$ in every realization of $\widehat{P}$.
\end{theorem}

The remainder of this section is concerned with the proof of Theorem~\ref{thm:proj}. For the convenience of the reader, we retrace, in a higher generality and with improved quantitative bounds, Shephard's ideas that lead him to discover $3$-dimensional polytopes that are not subpolytopes of stacked polytopes \cite{Shephard74}. We recall some notions.

\begin{definition}
A polytope is \Defn{$\kF{k}$-stacked} if it is the connected sum (cf.\ \cite[Sec.\ 3.2]{RG}) of polytopes with at most $k$ facets each. With this, a $\kF{(d+1)}$-stacked $d$-dimensional polytope is simply a \Defn{stacked} polytope.
\end{definition}
\vskip -2.4mm
\begin{figure}[htpb]
\centering
\begin{subfigure}[t]{.41\linewidth}
\centering
\includegraphics[width=.41\linewidth]{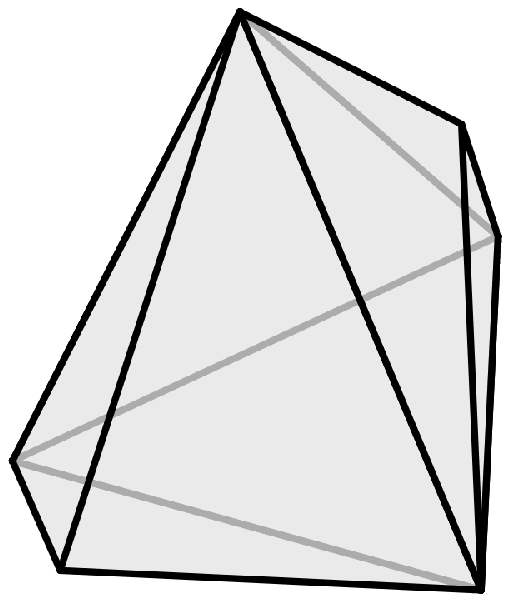}\qquad
\includegraphics[width=.41\linewidth]{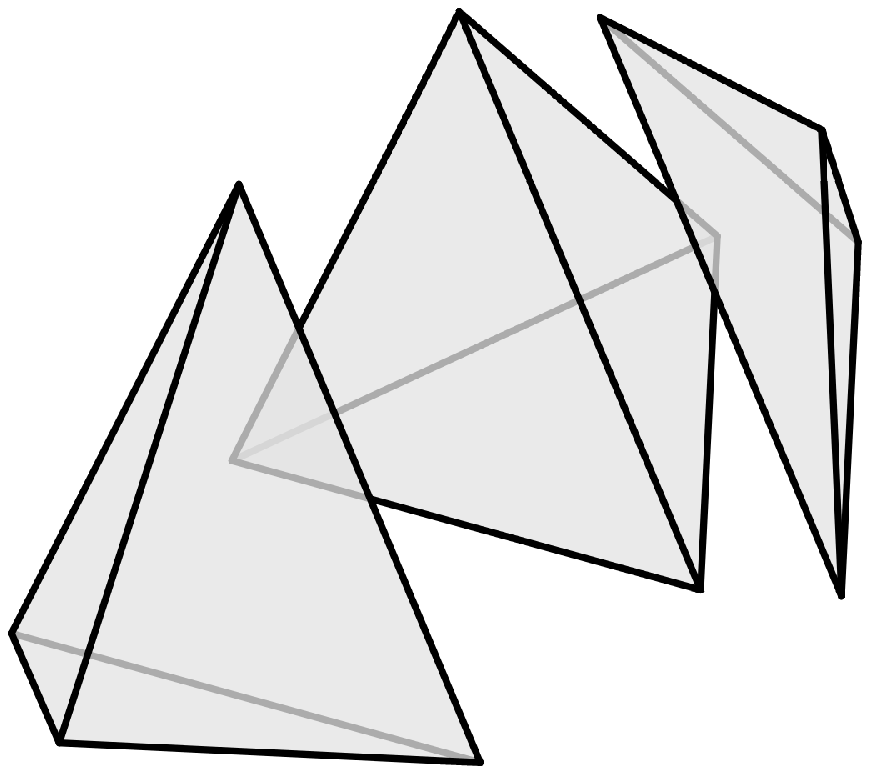}
\caption{$\kF{4}$-stacked}\label{fig:4stacked}
\end{subfigure}\qquad\qquad
\begin{subfigure}[t]{.41\linewidth}
\centering
\includegraphics[width=.41\linewidth]{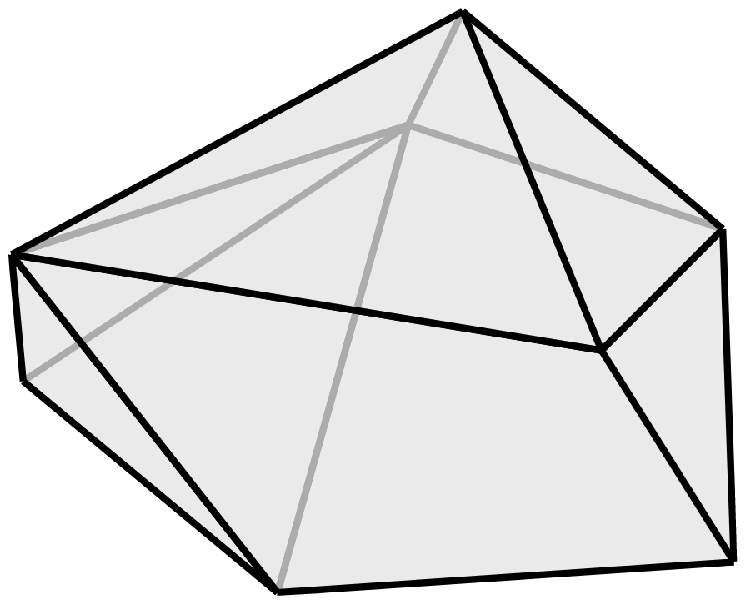}\qquad
\includegraphics[width=.41\linewidth]{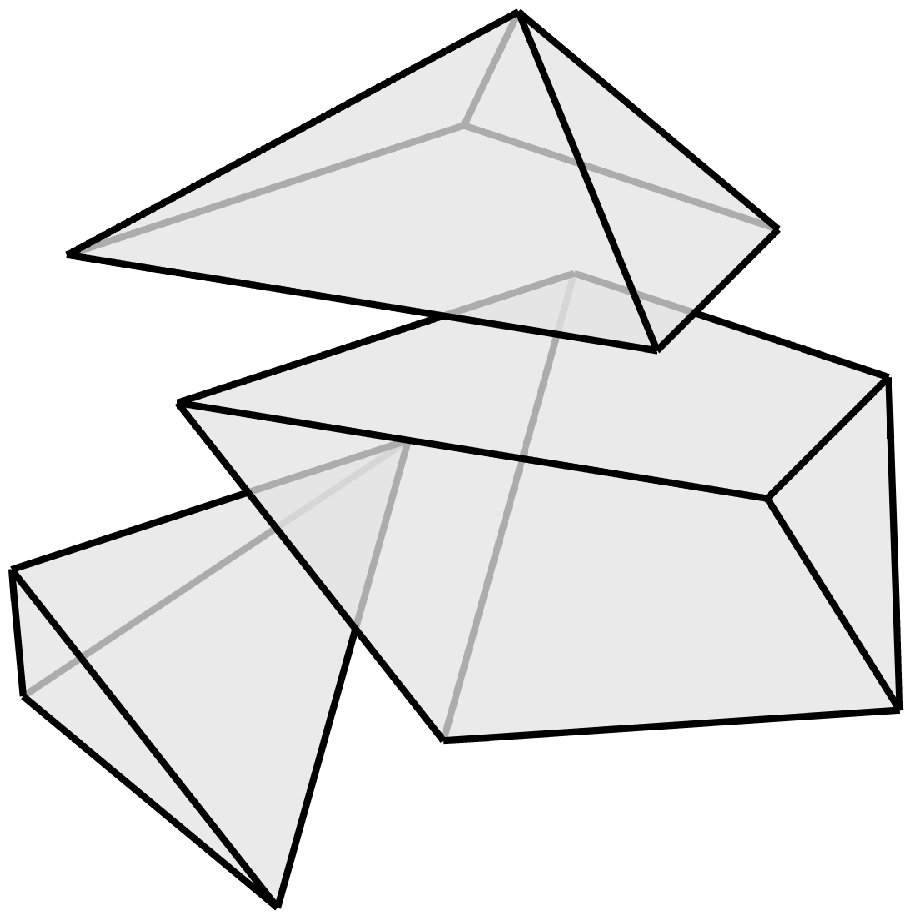}
\caption{$\kF{5}$-stacked}\label{fig:5stacked}
\end{subfigure}
\caption{A $\kF{4}$-stacked $3$-polytope (i.e.\ a stacked $3$-polytope) and a $\kF{5}$-stacked $3$-polytope.}\label{fig:stacked}
\end{figure}

Let $\dH(\cdot,\cdot)$ denote the Hausdorff distance between compact convex subsets of $\R^d$, cf.\ \cite[Sec.\ 1.8]{Schneider}. Let $B_r(x)$ denote the metric ball in $\R^d$ with center $x$ and radius $r$. Shephard's main observations are:

\begin{lemma}[cf.\ {\cite[$\mathbf{2.}$(i) {\&} (ii)]{Shephard74}}]\label{lem:dist}
Let $P$ be any polytope with $\dH(P, B_1(0))\le \e\le\frac{1}{36}$. Then, for any polytope $P'$ with $\F_0(P)\subset \F_0(P')$, we have $\dH(P', B_1(0))\le 6\sqrt{\e}$. 
\end{lemma}
Equivalently, if $Q'$ is any polytope with $\dH(Q', B_1(0))> \delta$, $0<\delta<1$, then $\dH(Q, B_1(0))>\frac{\delta^2}{36}$ for any subpolytope $Q$ of $Q'$.
\begin{proof}[Sketch of Proof]
By assumption, we have $B_{1-\e}(0)\subset P \subset  B_{1+\e}(0)$. Hence, every face of $P$ can be enclosed in some ball of radius $\sqrt{(1+\e)^2-(1-\e)^2}=2\sqrt{\e}$ and so every point in $\partial P$ is at distance at most $2\sqrt{\e}$ from some vertex of~$P$. 
Let now $v$ be any vertex of $P'$ not in $P$, and let $v_P$ denote the point of intersection of the line segment $\conv \{0,v\}$ with $\partial P$. Then $ \conv (\{v\}\cup  B_{1-\e}(0))\subset \conv (\{v\}\cup P)$, and since $\conv (\{v\}\cup P)$ contains no vertex of $P$ in the interior, we have \[\conv B_{2\sqrt{\e}}(v_P) \not\subset \conv (\{v\}\cup  B_{1-\e}(0)).\]
If $2\sqrt{\e}<1-\e$, this can be used to estimate the euclidean norm of $v$ as $||v||_2\le \frac{(1+\e)(1-\e)}{1-\e-2\sqrt{\e}}$, and hence if $\sqrt{\e}\le\nicefrac{1}{6}$, then $||v||_2\le1+6\sqrt{\e}$. This gives $B_{1-\e}(0)\subset P' \subset B_{1+6\sqrt{\e}}(0)$, or $\dH(P', B_1(0))\le6\sqrt{\e}$.
\end{proof}

\begin{lemma}[cf.\ {\cite[$\mathbf{2.}$(iii) {\&} (iv)]{Shephard74}}]\label{lem:dist2} 
For any $\kF{k}$-stacked $d$-dimensional polytope $S$, $d\ge 3$, we have \[\dH(S,B_1(0))\ge 2^{-2k-4}.\]
\end{lemma}

\begin{proof}[Sketch of Proof]
Assume $\dH(S,B_1(0))\le (1-\nicefrac{1}{\sqrt 2})$. As observed in the proof Lemma~\ref{lem:dist}, the edges of $S$ have length at most $4\sqrt{\dH(S,B_1(0))}$. 

Now, the polytope $S$ can be written as the connected sum of polytopes $S_i$, each of which has at most $k$ facets. Let $\varSigma$ be any one of the $S_i$ that contains the origin. We claim that $\varSigma$ has an edge of length at least $2^{-k}$. Indeed, since $\varSigma$ contains the origin, is has two vertices $v,w$ that enclose an angle at least $\nicefrac{\pi}{2}$ with respect to the origin. Furthermore, $\F_0(\varSigma)\cap B_{\nicefrac{1}{\sqrt 2}}(0)\subset \F_0(S)\cap B_{\nicefrac{1}{\sqrt 2}}(0)=\emptyset$, so these two vertices are at least at distance $1$ from each other. Since the graph of each polytope is connected, there must be a path of edges in $\varSigma$ from $v$ to $w$, and so one of these edges must be of length $(f_1(\varSigma))^{-1}$ or more. Finally, Sperner's Theorem shows that a polytope with $k$ facets has at most $\binom{k}{[\nicefrac{k}{2}]}\le2^k$ edges, so that $(f_1(\varSigma))^{-1}\ge 2^{-k}$, which gives the desired bound.
To combine the two observations, notice that since $d\ge 3$, all edges of $\varSigma$ are edges of $S$, so 
\[4\sqrt{\dH(S,B_1(0))}\ge2^{-k}\Longrightarrow \dH(S,B_1(0))\ge2^{-2k-4}.\qedhere\]
\end{proof}
 Combining the two lemmas above, we recover Shephard's main result.

\begin{cor}[cf.\ {\cite{Shephard74}}]\label{cor:sh} 
For any subpolytope $P$ of a $\kF{k}$-stacked $d$-dimensional polytope, $d\ge 3$,  we have \[\dH(P,B_1(0))\ge 2^{-4k-10}\cdot3^{-2}.\]
\end{cor}

In particular, every $d$-polytope that approximates $B_1(0)$ closely is not the subpolytope of any stacked polytope. We now only need to add a simple observation to Shephard's ideas:

\begin{prp}\label{prp:ratsub}
For $P$ is a subpolytope of a $\kF{k}$-stacked polytope $S$, then any face $\sigma$ of $P$ is a subpolytope of a $\kF{k}$-stacked polytope as well.
\end{prp}

\begin{proof}
It suffices to prove this in the case where $\sigma$ is a facet of $P$. Let $H$ denote the hyperplane spanned by $\sigma$. Recall that $S$ is obtained as the connected sum of polytopes $S_1,\dots, S_n$, and so $H\cap S$ is the connected sum of the polytopes $H\cap S_1,\dots, H\cap S_n$. Now every single one of the polytopes $S_i$ has most $k$ facets, and every facet of $H\cap S_i$ is obtained as the intersection of a facet of $S_i$ with $H$, so $H\cap S$ is $\kF{k}$-stacked. Observing that $\sigma=H\cap P$ is a subpolytope of $H\cap S$ finishes the proof.
\end{proof}

\paragraph*{Conclusion of proof}
\begin{proof}[{\bf Proof of Theorem~\ref{thm:proj}}]
Let $P$ be any $3$-dimensional polytope with $ \dH(P,B_1(0))<2^{-4\cdot6-10}\cdot 3^{-2}$. By Corollary~\ref{cor:sh}, $P$ is not a subpolytope of any $\kF{6}$-stacked polytope, and the same holds for any polytope projectively equivalent to $P$. 

Theorem~\ref{thm:proj} now provides a polytope $\widehat{P}$ of dimension $5$ that contains a face $F$ that is projectively equivalent to $P$ in every realization of $\widehat{P}$. Assume now that some polytope $O$ combinatorially equivalent to $\widehat{P}$ is a subpolytope of some stacked polytope. By Proposition~\ref{prp:ratsub}, any face of $O$ is a subpolytope of some $\kF{6}$-stacked polytope. But the face of $O$ corresponding to $F$ is projectively equivalent to $P$, and hence not obtained by deleting vertices of a $\kF{6}$-stacked polytope. A contradiction.
\end{proof}

{\small\bibliographystyle{myamsalpha}
\bibliography{Substacked}}

\end{document}